\documentclass[a4paper,fleqn]{cas-dc}

\usepackage[numbers]{natbib}
\usepackage{xcolor}
\usepackage{tikz}
\usetikzlibrary{arrows.meta}
 \usepackage{import}
\def\tsc#1{\csdef{#1}{\textsc{\lowercase{#1}}\xspace}}
\tsc{WGM}
\tsc{QE}
\tsc{EP}
\tsc{PMS}
\tsc{BEC}
\tsc{DE}

\newtheorem{theorem}{Theorem}
\newdefinition{assumption}{Assumption}
\newtheorem{corollary}{Corollary}
\newdefinition{remark}{Remark}
\newproof{proof}{Proof}
\newproof{pot}{Proof of Theorem \ref{thm}}

\begin{document}
\let\WriteBookmarks\relax
\def\floatpagepagefraction{1}
\def\textpagefraction{.001}
\shorttitle{Adaptive Average Current-Mode Control for Plasma Arc Loads}
\shortauthors{L. Ecker et~al.}

\title [mode = title]{Adaptive Average Current-Mode Control of a Buck Converter Supplying Plasma Arc Loads}
\tnotemark[1]

\tnotetext[1]{This document presents results from the LIGHTBOW project, which was funded by the Austrian Research Promotion Agency (FFG) as part of the Expedition Zukunft Wissenschaft program.}


\author[1,2]{Lukas Ecker}[
                        ]
\cormark[1]
\ead{lukas.ecker@jku.at}
\ead[url]{https://www.jku.at/en/institute-of-control-systems/}

\credit{Conceptualization, Methodology, Software, Formal analysis, 
Validation, Visualization, Writing -- original draft}

\affiliation[1]{organization={Johannes Kepler University Linz},
                addressline={Altenbergerstrasse 69}, 
                city={Linz},
                postcode={4040}, 
                state={Upper Austria},
                country={Austria}}

\author[1,3]{Thomas Voglhuber-Brunnmaier}[]
\ead{thomas.voglhuber-brunnmaier@jku.at}
\ead[URL]{https://www.jku.at/en/institute-for-microelectronics-and-microsensors/}

\credit{Conceptualization, Methodology,
Writing -- review and editing}



\affiliation[2]{organization={Institute of Control Systems}}
\affiliation[3]{organization={Institute for Microelectronics and Microsensors}}



\cortext[cor1]{Corresponding author}

\begin{abstract}
Reliable operation of plasma processes requires a power supply that can
regulate and limit the arc current despite rapid changes of the electrical
load. Gas-flow fluctuations, arc motion, changing attachment conditions, and
electrode interactions can cause abrupt arc-voltage variations and may lead
to arc interruption. The power supply must therefore provide fast current
control and a suitable actuation layer for supervisory stabilization. This
paper presents an adaptive average current-mode controller for a buck converter
supplying a nonlinear plasma arc load. Based on an averaged model of the
converter inductor dynamics, the controller combines current-error feedback
with feedforward compensation of the measured arc voltage. Online adaptation of the inductance and its effective series resistance compensates for converter-parameter uncertainty without requiring an online nonlinear arc model. A
Lyapunov analysis establishes bounded estimation errors and asymptotic current
tracking for the averaged closed-loop system under bounded-reference
assumptions. Experiments with an argon plasma arc demonstrate contact-opening
ignition and current-reference tracking under varying arc-voltage
conditions, supporting the controller as a fast current-control layer for
plasma applications.

\end{abstract}

\begin{graphicalabstract}
\includegraphics{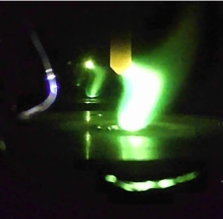}
\end{graphicalabstract}

\begin{highlights}
\item Adaptive current control of a buck converter supplying a plasma arc
\item Measured arc-voltage feedforward avoids an online nonlinear arc model
\item Online adaptation compensates uncertainty in converter inductance and series resistance
\item Experimental validation includes contact-opening ignition and current tracking
\end{highlights}


\maketitle

\section{Introduction}
\label{sec:intro}

Plasma arc power supplies typically have to operate under rapid,
nonlinear, and partly unpredictable changes in the electrical load
while respecting strict current and voltage limits of the converter,
electrodes, and the overall process
\cite{Gerdeman1972,Hassan2025}.
These requirements are particularly demanding during contact-opening
ignition, when the load changes abruptly from a low-resistance contact
to a nonlinear plasma discharge and the arc voltage rises rapidly.

During established arc operation, thermal effects, gas-flow
fluctuations, turbulence, changing attachment conditions, electrode
condition, and external electromagnetic influences can further perturb
the electrical operating point. This is particularly relevant in
hydrogen plasma smelting reduction (HPSR), where a hydrogen-containing
plasma arc provides both the thermal energy for melting and reactive
hydrogen species for the reduction of iron oxides. Such
process-dependent variations can significantly affect arc voltage and
stability \cite{Ernst2023}. A current controller for such a
load must therefore provide fast tracking and disturbance rejection
without relying on fixed plasma parameters.

Line-commutated thyristor rectifiers remain an established power-supply
technology in metallurgical high-current applications such as DC arc
furnaces and plasma systems \cite{Solanki2015}, and are also used in
existing HPSR research facilities \cite{Ernst2023}. Their control
action is, however, inherently linked to the mains frequency, which
limits the achievable electrical actuation bandwidth. This limitation
becomes particularly relevant for HPSR, where rapid and pronounced
arc-voltage fluctuations can challenge stable plasma operation
\cite{AlNasser2026}. High-frequency switched-mode converter stages
therefore provide an attractive alternative for implementing a faster
current-control layer.

The performance of model-based converter control also depends on the
accuracy of the electrical parameters used in the control law. In
particular, the effective inductance and series resistance may deviate
from their nominal values due to component tolerances, temperature
variations, magnetic operating conditions, and parasitic effects.
Adaptive and robust control approaches for DC--DC converters have
therefore been investigated to account for parameter uncertainty and
changing operating conditions
\cite{Nizami2021,Saadat2022,ElFadil2009}.

For plasma arc loads, the converter-side parameter uncertainty is
combined with a strongly nonlinear and time-varying load. Explicitly
incorporating the plasma characteristics into the real-time current
controller would require a sufficiently accurate arc model and
knowledge or online identification of its operating-condition-dependent
parameters. This motivates a control structure that separates the
uncertain converter dynamics from the nonlinear plasma-load behavior.

This work addresses the current control of a buck-converter-based
plasma arc power supply. The proposed adaptive controller combines
feedback of the inductor-current tracking error with feedforward
compensation of the measured arc voltage. By using the measured
$U_{\mathrm{arc}}$ directly, the nonlinear plasma load is treated as
a measurable load-dependent disturbance. The uncertain inductance and effective series resistance are adapted online within the control law.
Consequently, the implemented control law does not require real-time
evaluation of the nonlinear voltage--current characteristic of the arc.
A supervisory layer supplies a bounded and rate-limited current
reference and coordinates ignition and protection functions. The main contributions of this work are:
\begin{itemize}
\item[(i)] an implementation-oriented adaptive current-control architecture
for a buck-converter-based plasma power supply that directly compensates
the measured arc voltage without requiring an online nonlinear plasma
model;

\item[(ii)] online compensation of uncertainty in the converter inductance
and effective series resistance, together with a Lyapunov-based analysis
of the resulting tracking and parameter-error dynamics;

\item[(iii)] experimental validation on a laboratory plasma-arc setup during contact-opening ignition, current-reference variations, and rapid arc-voltage disturbances.
\end{itemize}
\section{System description and averaged model}
\label{sec:model}

This section describes the plasma arc load and the buck-converter power stage, derives the averaged inductor model used for controller design, and introduces the implemented current-control architecture.

\subsection{Electric arc model}
\label{subsec:arc_circuit}

The electric arc, which occurs in various industrial processes---including
hydrogen plasma smelting reduction (HPSR), electric arc furnaces (EAFs),
plasma arc welding (PAW), plasma cutting, and plasma spraying---is a highly
nonlinear thermo-electrical plasma phenomenon.

Its detailed physical description involves ionization dynamics,
temperature-dependent conductivity, magnetohydrodynamic effects,
gas flow, and electrode interactions \cite{Beilis2020,Haraldsson2025}.
High-fidelity descriptions therefore typically require coupled
multiphysics models involving electromagnetic, fluid-dynamic, and
thermal phenomena and are generally too complex for the
implementation-oriented current-control design considered in this work.

For physical interpretation and offline simulation, the arc is therefore represented here by a reduced electrical model based on the static voltage--current (VI) characteristic proposed by \citet{andrea2015arc}. Dynamic plasma phenomena and effects not captured
by this quasi-static representation are treated as unmodeled load
variations or external disturbances.

\paragraph{Electrode gap and arc length.}
In the configuration considered in this work, the electric arc is formed
between a vertically movable electrode and a grounded metallic contact
plate, as illustrated in Fig.~\ref{fig:electrode_arc_setting}. The geometric
electrode gap $g(t)$ is obtained directly from the measured position of
the electrode axis, whereas the actual arc length
$\ell_{\mathrm{arc}}(t)$ corresponds to the path length of the plasma column between its attachment points. Due to electromagnetic forces, buoyancy, gas flow, and other
process-dependent effects, the arc is generally curved
\cite{Beilis2020,Haraldsson2025,VoglhuberBrunnmaier2026}, such that
\(
\ell_{\mathrm{arc}}(t) \geq g(t).
\)

In the laboratory setup, $\ell_{\mathrm{arc}}(t)$ can additionally be
estimated from synchronized camera images. However, this optical
information is deliberately not used for current control, since reliable
online reconstruction of the arc trajectory is often difficult in
industrial plasma processes due to limited optical access and harsh
environmental conditions. The proposed controller therefore does not
require $\ell_{\mathrm{arc}}(t)$ as an online measurement.

\begin{figure}[t]
\centering
\includegraphics[width=\columnwidth]{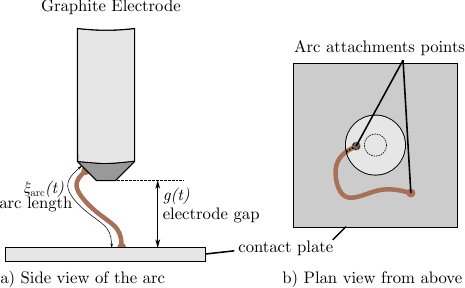}
\caption{Schematic representation of the electrode--plate geometry.
(a)~Side view illustrating the geometric electrode gap $g(t)$ and the
arc length $\ell_{\mathrm{arc}}(t)$ along the curved plasma column.
(b)~Top view showing the projected arc trajectory between the attachment
points.}
\label{fig:electrode_arc_setting} 
\end{figure}

\paragraph{Static arc characteristic.}

From a macroscopic viewpoint, the electrical behavior of the considered
process can be characterized by three distinct states: a low-resistance
conductive state, a high-resistance non-arcing state, and an arc-dominated
state following ignition. Following \citet{andrea2015arc}, the non-arcing
and arc-dominated states can be represented by a static discharge
characteristic that combines a high-resistance branch with a nonlinear
arc characteristic.

In the experiments considered later, arc ignition is initiated by contact
opening. While the electrodes are in contact, the current flows through a
low-resistance metallic conduction path, which may be approximated as
\begin{equation}
U_{\mathrm{arc}}
\approx
R_{\mathrm{contact}} I_{\mathrm{arc}},
\label{eq:contact_branch}
\end{equation}
where $R_{\mathrm{contact}}$ denotes the effective resistance of the closed
electrical contact. As the electrode is retracted and the gap $g(t)$
increases, the metallic conduction path is interrupted and the voltage
across the electrode gap rises.

Following \citet{andrea2015arc}, the non-arcing state is represented by a
high-resistance branch,
\begin{equation}
U_{\mathrm{arc}}
\approx
R_c I_{\mathrm{arc}},
\label{eq:nonarcing_branch}
\end{equation}
where $R_c$ denotes an effective resistance associated with the non-arcing
state and is substantially larger than the metallic contact resistance,
i.e.,
\begin{equation}
R_{\mathrm{contact}} \ll R_c.
\label{eq:contact_resistance_relation}
\end{equation}
Thus, $R_c$ should not be interpreted as the resistance of the closed
electrical contact. As the voltage across the electrode gap increases and
reaches the corresponding disruption voltage, the discharge transitions
to the arc-dominated state.

For an established arc, the static arc voltage is approximated following
\citet{andrea2015arc} as
\begin{equation}
U_{\mathrm{arc}}
=
\frac{\alpha}
{\arctan\!\bigl(\beta I_{\mathrm{arc}}\bigr)},
\label{eq:andrea_static}
\end{equation}
where $I_{\mathrm{arc}}$ denotes the arc current. In general, the
parameters $\alpha>0$ and $\beta>0$ depend on the arc length, gas pressure,
gas composition, electrode material, and further operating conditions.

In the present work, the discharge is investigated in an argon atmosphere
at nominally constant chamber pressure, using a graphite electrode and a
steel contact plate. Accordingly, only the explicit dependence on
$\ell_{\mathrm{arc}}$ is retained in the following, while the influence of
the remaining operating conditions is absorbed into coefficients identified
for this particular setup.\footnote{Classical empirical arc models include
power-law resistance relations such as the Warrington formula
\citep{warrington1931reactance,warrington1968protective}. A more recent
alternative is the static voltage--current representation proposed by
\citet{andrea2015arc}, which is adopted here for the present analysis.}

Following \citet{andrea2015arc}, the dependence on arc length is modeled as
\begin{subequations}
\begin{align}
\alpha(\ell_{\mathrm{arc}})
&=
\gamma\,\ell_{\mathrm{arc}}+\delta,
\label{eq:alpha_length}
\\
\beta(\ell_{\mathrm{arc}})
&=
\beta_0+\eta\,\ell_{\mathrm{arc}},
\label{eq:beta_length}
\end{align}
\end{subequations}
with $\gamma>0$, $\delta>0$, $\beta_0>0$, and $\eta\geq0$. For the
investigated argon setup, the experimental VI characteristics were well
approximated using
\(
\gamma = 2\,\mathrm{V\,mm^{-1}},\,
\delta = 15\,\mathrm{V},\,
\beta_0 = 0.18\,\mathrm{A^{-1}},
\)
and
\(
\eta = 0.004\,\mathrm{A^{-1}mm^{-1}}.
\)
\footnote{These parameter values are representative of the investigated
setup. Noticeable run-to-run variations were observed due to thermally
induced changes, electrode surface conditions, varying arc attachment,
and other process-dependent effects.}

Increasing $\ell_{\mathrm{arc}}$ shifts the VI characteristic toward higher arc voltages, predominantly through the increase in \(\alpha\), while the variation of \(\beta\) modifies its curvature. To obtain a continuous
static representation covering both the high-resistance non-arcing state
and established arc operation, the two branches are combined following
\citet{andrea2015arc} as
\begin{equation}
U_{\mathrm{arc}}
=
\frac{
\alpha(\ell_{\mathrm{arc}})\,R_c\,I_{\mathrm{arc}}
}{
\arctan\!\bigl(\beta(\ell_{\mathrm{arc}})\,I_{\mathrm{arc}}\bigr)
I_{\mathrm{arc}}R_c
+
\alpha(\ell_{\mathrm{arc}})
}.
\label{eq:andrea_discharge}
\end{equation}

For small $|I_{\mathrm{arc}}|$, \eqref{eq:andrea_discharge} approaches the
high-resistance branch \eqref{eq:nonarcing_branch}. For sufficiently large
arc currents, the arc resistance becomes small compared with $R_c$, such
that the influence of the high-resistance branch becomes negligible and
the model approaches the pure-arc relation \eqref{eq:andrea_static}.

For established arc operation, an effective static arc resistance can be
defined as
\begin{equation}
R_{\mathrm{arc}}(I_{\mathrm{arc}},\ell_{\mathrm{arc}})
=
\frac{U_{\mathrm{arc}}}{I_{\mathrm{arc}}}
=
\frac{
\alpha(\ell_{\mathrm{arc}})
}{
\arctan\!\bigl(\beta(\ell_{\mathrm{arc}})\,I_{\mathrm{arc}}\bigr)
I_{\mathrm{arc}}
}.
\label{eq:andrea_Rarc}
\end{equation}

After ignition, continued electrode retraction generally increases the arc
length and therefore shifts the static VI characteristic through
\eqref{eq:alpha_length}--\eqref{eq:beta_length}. Figure~\ref{fig:arc_vi}
illustrates this effect for representative values of
$\ell_{\mathrm{arc}}$.

\begin{figure}[htbp]
\centering
\includegraphics[width=\columnwidth]
{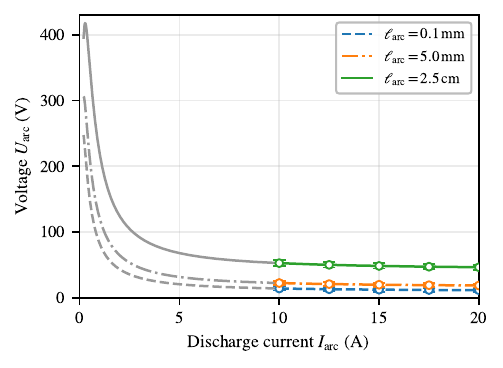}
\caption{Combined static discharge characteristic
\eqref{eq:andrea_discharge} for the argon test setup at different
established arc lengths during electrode retraction. Increasing
$\ell_{\mathrm{arc}}$ shifts the VI characteristic toward higher arc
voltages. Curve segments below $10\,\mathrm{A}$ are shown in gray because
this current range is not traversed during the investigated contact-opening
sequence.}
\label{fig:arc_vi}
\end{figure}

\paragraph{Implications for controller design.}

Contact-opening ignition represents the dominant load transient considered
in the proposed control architecture. The discharge changes from an
approximately resistive contact state to a strongly nonlinear plasma load,
while the arc voltage $U_{\mathrm{arc}}$ can rise rapidly and the effective
arc resistance $R_{\mathrm{arc}}$ may vary substantially.

For current control, these load variations are not represented by an
online plasma model. Instead, the measured arc voltage is treated as a
directly available load-dependent disturbance and compensated through
feedforward, while the adaptive controller regulates the inductor current
$I_L$. The quasi-static relations
\eqref{eq:andrea_discharge}--\eqref{eq:andrea_Rarc} are therefore used only
for physical interpretation, offline simulation, and setpoint planning.

\subsection{Power converter stage and measurements}
\label{subsec:power_stage}

Power supplies for plasma arc loads must accurately regulate the discharge
current while remaining responsive to rapid and nonlinear load variations.
During contact-opening ignition, the current must be limited to protect the
power-electronic components and electrodes. During established arc
operation, fast electrical actuation is required to maintain a stable operating
point and to reduce the risk of arc extinction under pronounced load-voltage
variations.

\paragraph{Conventional arc power supplies.}

In industrial high-power applications, arc power supplies are commonly
implemented using line-commutated thyristor rectifiers with phase-angle
control \cite{Solanki2015}. These converters provide robust and mature
solutions at very high power levels, but their control action is inherently
linked to the mains frequency. Consequently, the rate at which the converter
output voltage and arc current can be adjusted is limited. This restricted
actuation bandwidth can become critical during ignition and under rapid
arc-voltage variations, particularly in operating points where arc stability
is sensitive to changing process conditions \cite{Zarl2020}.

\paragraph{DC step-down converter.}

PWM-controlled buck converters supplied by a stiff DC link provide a suitable
alternative when faster electrical actuation is required. Compared with
line-commutated thyristor rectifiers, their control action is not directly
constrained by the mains period, allowing a substantially higher achievable
current-control bandwidth and a faster response to load variations
\cite{Solanki2015}.
For higher power levels, several buck-converter phases can in principle be
connected in an interleaved configuration. Such an architecture distributes
the total current among multiple semiconductor switches and inductors and can
reduce the output-current ripple while preserving the fast local dynamics of
the individual phases \cite{Zhang2025Interleaved}. A multiphase realization,
however, introduces additional requirements regarding current sharing,
synchronization, and phase-to-phase parameter variations
\cite{Li2011CurrentSharing,Sinafar2023}.
The present work considers a single buck-converter stage as the fundamental
building block of such an architecture. The interleaved extension itself is
outside the scope of this study, whereas the proposed control and adaptation
concepts are formulated at the phase level and may therefore provide a basis
for a future multiphase implementation.

\paragraph{Buck topology and measurements.}

The buck-converter topology considered in this work is shown in Fig.~\ref{FIG:1}. A semiconductor switch and a freewheeling diode connect the DC-link supply to an output network consisting of a series inductor and a capacitor connected in parallel with the plasma arc. The plasma discharge constitutes the nonlinear and time-varying converter load. 
\begin{figure}[h]
\centering
\includegraphics[width=.99\columnwidth]{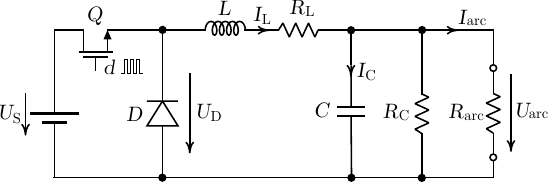}
\caption{Schematic of the buck-converter power stage supplying the plasma arc load.}
\label{FIG:1}
\end{figure}
The proposed controller tracks an inductor-current reference and limits the
discharge current during contact-opening ignition. The relevant converter-side
uncertainties are the inductance $L$ and its effective series resistance
$R_L$, which may vary due to component tolerances, temperature-dependent
losses, magnetic operating conditions, and aging. These parameters are
adapted online by the equations introduced in
Section~\ref{sec:current_control}. The measured feedback quantities are the
inductor current $I_L$ and the arc voltage $U_{\mathrm{arc}}$.
Because the inductor supplies the complete output network, $I_L$ differs
from the arc current $I_{\mathrm{arc}}$ by the current flowing into the
capacitor and parallel leakage branch.\footnote{According to the output-node
current balance, the inductor current is shared between the plasma arc, the
output capacitor, and the parallel leakage path,
\(
I_L = I_{\mathrm{arc}} + I_C + I_{R_C}.
\)
If the capacitor and leakage currents are small compared with $I_L$, then
$I_{\mathrm{arc}} \approx I_L$. During rapid arc-voltage transients, this
approximation may no longer hold; the full current balance is therefore
retained in Sec.~\ref{subsec:averaged_model}.}

\subsection{Averaged Buck Converter Model}
\label{subsec:averaged_model}

The buck converter in Fig.~\ref{FIG:1} operates in continuous conduction mode
(CCM), i.e., the inductor current remains positive throughout the switching
period. The CCM assumption is consistent with the investigated arc operating
conditions, for which a nonzero current reference is maintained and the series
inductance limits current ripple and rapid current decay.

The switched dynamics are replaced by their average over one PWM period following standard averaging theory \cite{khalil2002,sanders1991}, neglecting switching ripple, dead-time effects, and switching losses. When the semiconductor conducts, the inductor is connected to the DC-link voltage $U_S$; when it blocks, the current freewheels through the diode.
Averaging the two configurations over one switching period introduces the duty cycle $d(t)\in[0,1]$ as the primary control input and yields the averaged converter input voltage
\begin{equation}
U_D(t) = U_S\, d(t).
\label{eq:UD_avg}
\end{equation}

\paragraph{Inductor dynamics.}
Applying Kirchhoff's voltage law to the inductor branch in Fig.~\ref{FIG:1} gives 
\begin{equation}
U_D = L\,\dot I_L + R_L I_L + U_{\mathrm{arc}}.
\label{eq:UD}
\end{equation}
This relation serves as the reduced plant model used for the adaptive current-controller design in Section~\ref{sec:current_control}: the duty cycle acts on $U_D$, while $U_{\mathrm{arc}}$ enters as a measurable disturbance that is compensated by feedforward.

\paragraph{Voltage dynamics.}
Applying Kirchhoff's current law at the converter output node gives
\begin{equation}
I_L
=
C\,\dot U_{\mathrm{arc}}
+
\frac{U_{\mathrm{arc}}}{R_C}
+
I_{\mathrm{arc}},
\label{eq:IL_arc}
\end{equation}
where the inductor current is divided between the output capacitor,
the parallel leakage resistance, and the plasma discharge. Using the effective arc resistance introduced in
\eqref{eq:andrea_Rarc}, the discharge current satisfies
\begin{equation}
I_{\mathrm{arc}}
=
\frac{U_{\mathrm{arc}}}
{R_{\mathrm{arc}}(I_{\mathrm{arc}},\ell_{\mathrm{arc}})}.
\label{eq:Iarc_Rarc}
\end{equation}
\paragraph{Use of the model for simulation and control.}
The static discharge model in Subsection~\ref{subsec:arc_circuit} and the output-node balance \eqref{eq:IL_arc} are used for physical interpretation, offline simulation, and setpoint planning. For real-time current control, the nonlinear arc resistance is not evaluated explicitly. Instead, the controller is derived directly from the inductor dynamics \eqref{eq:UD} and uses the measured $U_{\mathrm{arc}}$ as a feedforward signal. The uncertain converter parameters $L$ and $R_L$ are estimated by the adaptive current controller developed in Section~\ref{sec:current_control}. Consequently, only \eqref{eq:UD_avg} and \eqref{eq:UD} are required for the controller design and stability analysis.

\subsection{Adaptive current-control architecture}
\label{subsec:control_architecture}

The proposed adaptive current-control architecture is shown in
Fig.~\ref{FIG:2}. A supervisory layer (Level~2) coordinates process-level
functions, including contact detection, ignition sequencing, reference
generation, protection logic, camera-based monitoring, and control of the
electrode-positioning axis. The detailed design of these supervisory functions
is outside the scope of this work. The generated current reference is shaped by
magnitude and rate limiters before being passed to the adaptive current
controller (Level~1) as $I_{L,\mathrm{ref}}$ and
$\dot I_{L,\mathrm{ref}}$.

The inner current-control layer combines current-error feedback with
feedforward compensation of the measured $U_{\mathrm{arc}}$ and online
adaptation of the uncertain converter parameters $L$ and $R_L$. The nonlinear
static arc model \eqref{eq:andrea_static}--\eqref{eq:andrea_Rarc} is therefore
not evaluated online. Instead, the measured arc voltage enters the control law
directly, while the resulting converter-voltage command is translated into the
duty cycle $d(t)$ by the PWM modulator. The detailed control and adaptation
laws are developed in Section~\ref{sec:current_control}.

The shaped current reference and the duty cycle are constrained according to $d_{\max}\le 1$ and 
\begin{equation}
I_{L,\mathrm{ref}}(t)\in[I_{L,\min},\,I_{L,\max}],
\quad
d(t)\in[0,d_{\max} ],
\label{eq:Iref_bounds}
\end{equation}
which, by \eqref{eq:UD_avg}, implies $U_D\in[0,U_S d_{\max}]$.  In addition, the
reference-shaping stage limits the rate of change of $I_{L,\mathrm{ref}}$,
so that both $I_{L,\mathrm{ref}}$ and $\dot I_{L,\mathrm{ref}}$ remain
bounded. These bounds reflect the physical operating limits of the converter
and process and motivate the assumptions used in the idealized stability analysis
in Section~\ref{sec:current_control}.

\begin{figure*}
  \centering
  \includegraphics[width=1.0\textwidth]
  {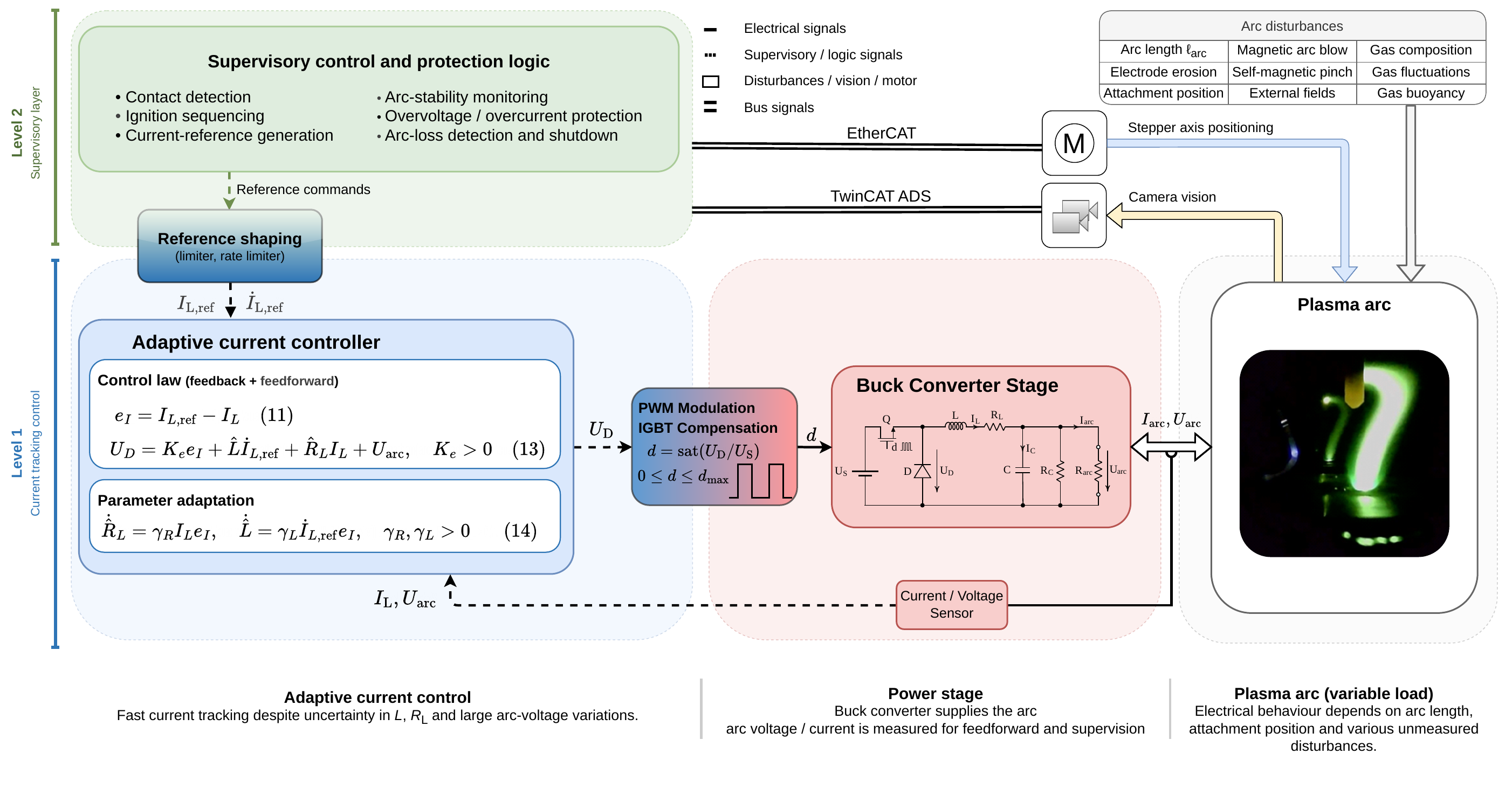}
  \caption{Adaptive current-control architecture for the buck-converter plasma
  arc power supply. The supervisory layer (Level~2) coordinates process-level
  functions and generates the current reference, while the fast current-control
  layer (Level~1) regulates the inductor current using the measured
  $I_L$ and $U_{\mathrm{arc}}$.}
  \label{FIG:2}
\end{figure*}

\section{Adaptive Current Tracking Controller}
\label{sec:current_control}

The adaptive current controller is designed to ensure accurate tracking of the
inductor current reference $I_{L,\mathrm{ref}}(t)$ in the presence of uncertain
converter parameters and time-varying arc conditions. The control law
incorporates online adaptation of the inductance $L$ and its effective series
resistance $R_L$, while the measured arc voltage $U_{\mathrm{arc}}$ is
compensated directly through feedforward.

The stability and convergence properties of the resulting closed-loop current
dynamics are analyzed using a Lyapunov-based approach. It is shown that the
current tracking error converges asymptotically to zero, while all error signals
of the ideal unsaturated averaged closed-loop system remain bounded.

\subsection{Adaptive Current-Control Law}
\label{subsec:adaptive_control_law}

The control objective is to track the bounded and continuously differentiable
inductor-current reference $I_{L,\mathrm{ref}}(t)$ provided by the
reference-shaping stage. As introduced in
Subsection~\ref{subsec:control_architecture}, both
$I_{L,\mathrm{ref}}(t)$ and $\dot I_{L,\mathrm{ref}}(t)$ are bounded.

The current tracking error is defined as
\begin{equation}
e_I(t)
:=
I_{L,\mathrm{ref}}(t)-I_L(t).
\label{eq:current_error}
\end{equation}
Differentiating \eqref{eq:current_error} and using the averaged inductor
dynamics \eqref{eq:UD} yields
\begin{equation}
\dot e_I
=
\dot I_{L,\mathrm{ref}}
-
\frac{1}{L}
\left(
U_D-R_L I_L-U_{\mathrm{arc}}
\right).
\label{eq:error_dynamics_open}
\end{equation}

To compensate the uncertain converter parameters while directly accounting
for the measured arc voltage, the adaptive control law is chosen as
\begin{equation}
U_D
=
K_e e_I
+
\hat L\,\dot I_{L,\mathrm{ref}}
+
\hat R_L I_L
+
U_{\mathrm{arc}},
\qquad K_e>0,
\label{eq:adaptive_control_law}
\end{equation}
where $\hat L$ and $\hat R_L$ denote the online estimates of the inductance
and effective series resistance, respectively. The corresponding adaptation
laws are
\begin{subequations}
\begin{align}
\dot{\hat R}_L
&=
\gamma_R I_L e_I,
\label{eq:R_adaptation}
\\
\dot{\hat L}
&=
\gamma_L \dot I_{L,\mathrm{ref}} e_I,
\label{eq:L_adaptation}
\end{align}
\end{subequations}
with adaptation gains $\gamma_R>0$ and $\gamma_L>0$.

Defining the parameter estimation errors as
\begin{equation}
e_{R_L}:=R_L-\hat R_L,
\qquad
e_L:=L-\hat L,
\label{eq:parameter_errors}
\end{equation}
and substituting \eqref{eq:adaptive_control_law} into
\eqref{eq:error_dynamics_open} gives the closed-loop error dynamics
\begin{subequations}
\begin{align}
\dot e_I
&=
-\frac{K_e}{L}e_I
+
\frac{I_L}{L}e_{R_L}
+
\frac{\dot I_{L,\mathrm{ref}}}{L}e_L,
\label{eq:error_dynamics_current}
\\
\dot e_{R_L}
&=
-\gamma_R I_L e_I,
\label{eq:error_dynamics_resistance}
\\
\dot e_L
&=
-\gamma_L \dot I_{L,\mathrm{ref}}e_I.
\label{eq:error_dynamics_inductance}
\end{align}
\label{eq:error_dynamics}
\end{subequations}

The resulting closed-loop error dynamics depend on the time-varying
current reference and the inductor current $I_L(t)$. Their boundedness
and convergence properties are analyzed in the following subsection.

\subsection{Stability Properties of the Current Loop}
\label{subsec:stability}

The stability properties of the adaptive current loop are analyzed for the
ideal unsaturated averaged closed-loop system. The analysis is based on the
following assumptions.

\begin{assumption}
The converter parameters $L$ and $R_L$ are constant but unknown, with
$L>0$ and $R_L\geq0$.
\end{assumption}

\begin{assumption}
The reference $I_{L,\mathrm{ref}}(t)$ is continuously differentiable, and
both $I_{L,\mathrm{ref}}(t)$ and $\dot I_{L,\mathrm{ref}}(t)$ are bounded.
\end{assumption}

\begin{assumption}
The averaged converter model \eqref{eq:UD} is valid and the commanded
converter voltage $U_D$ can be applied without actuator saturation. The arc
voltage $U_{\mathrm{arc}}$ used for feedforward compensation is assumed to be
measured without error or delay.
\end{assumption}

Under these assumptions, the following theorem establishes boundedness of the
closed-loop error signals and asymptotic convergence of the current tracking
error. Convergence of the parameter estimates to their true values is not
guaranteed in general and requires additional excitation conditions.
\begin{theorem}

Consider the closed-loop error system
\eqref{eq:error_dynamics} under Assumptions~1--3. Then the following
properties hold:

\begin{enumerate}

\item For any initial condition, the error signals
$e_I(t)$, $e_{R_L}(t)$, and $e_L(t)$ remain bounded for all $t\geq0$.

\item The current tracking error is square integrable and converges
asymptotically to zero, i.e.,
\begin{equation}
\int_0^\infty e_I^2(t)\,\mathrm{d}t < \infty,
\qquad
\lim_{t\to\infty} e_I(t)=0.
\label{eq:tracking_convergence}
\end{equation}

\end{enumerate}

\end{theorem}

\begin{proof}
Consider the Lyapunov candidate
\begin{equation}
\label{eq:V_candidate}
V
=
\frac{1}{2}L e_I^2
+
\frac{1}{2\gamma_R}e_{R_L}^2
+
\frac{1}{2\gamma_L}e_L^2,
\end{equation}
with $\gamma_R>0$ and $\gamma_L>0$. Since $L>0$, the function $V$
is positive definite and radially unbounded with respect to the error vector
\(
e:=\begin{bmatrix}e_I & e_{R_L} & e_L\end{bmatrix}^{\top}.
\)

Differentiating $V$ along the trajectories of the closed-loop error system
\eqref{eq:error_dynamics} gives
\begin{align}
\dot V
&=
L e_I\dot e_I
+
\frac{1}{\gamma_R}e_{R_L}\dot e_{R_L}
+
\frac{1}{\gamma_L}e_L\dot e_L
\nonumber\\
&=
L e_I
\left(
-\frac{K_e}{L}e_I
+
\frac{I_L}{L}e_{R_L}
+
\frac{\dot I_{L,\mathrm{ref}}}{L}e_L
\right)
\nonumber\\
&\quad
-
e_{R_L}I_L e_I
-
e_L\dot I_{L,\mathrm{ref}}e_I
\nonumber\\
&=
-K_e e_I^2
\leq 0.
\label{eq:Vdot_negative}
\end{align}

Hence, $V(t)$ is non-increasing and satisfies
$V(t)\leq V(0)$ for all $t\geq0$. Since $V$ is positive definite and
radially unbounded with respect to the error vector, the error signals
$e_I(t)$, $e_{R_L}(t)$, and $e_L(t)$ remain bounded for all $t\geq0$,
which proves statement~(i).

Furthermore, integration of \eqref{eq:Vdot_negative} over $[0,T]$ yields
\begin{equation}
V(T)-V(0)
=
-\int_0^T K_e e_I^2(t)\,\mathrm{d}t.
\end{equation}
Since $V(T)\geq0$,
\begin{equation}
\int_0^\infty e_I^2(t)\,\mathrm{d}t
\leq
\frac{V(0)}{K_e}
<\infty,
\label{eq:e_square_integrable}
\end{equation}
and hence the tracking error is square integrable.

It remains to establish asymptotic convergence. Since
$I_{L,\mathrm{ref}}$ and $e_I$ are bounded, the relation
\(
I_L=I_{L,\mathrm{ref}}-e_I
\)
implies that $I_L$ is bounded. Together with boundedness of
$e_{R_L}$, $e_L$, and $\dot I_{L,\mathrm{ref}}$, the closed-loop error
dynamics imply that $\dot e_I$ is bounded. Therefore,
\begin{equation}
\ddot V
=
-2K_e e_I\dot e_I
\label{eq:Vddot}
\end{equation}
is bounded, and thus $\dot V$ is uniformly continuous.

Since $V(t)$ is non-increasing and bounded from below, it converges to a
finite limit. Hence, by Barbalat's lemma \citep[Lemma~8.2]{khalil2002},
\(
\lim_{t\to\infty}\dot V(t)=0.
\)
Using \eqref{eq:Vdot_negative} and $K_e>0$ then yields
\(
\lim_{t\to\infty}e_I(t)=0,
\)
which proves statement~(ii).
\end{proof}

The preceding theorem guarantees asymptotic current tracking, but does not
generally imply convergence of the parameter estimates. A stronger result can
be obtained for the practically relevant case of a constant nonzero current
reference, for which the resistance estimate converges to its true value.

\begin{corollary}
\label{cor:RL_convergence}

Assume that the current reference is constant and nonzero,
\begin{equation}
I_{L,\mathrm{ref}}(t)\equiv \bar{I}_{L,\mathrm{ref}}\neq0,
\qquad
\dot I_{L,\mathrm{ref}}(t)\equiv0.
\end{equation}
Under the conditions of the preceding theorem, the resistance estimation error
converges asymptotically to zero, and hence
\begin{equation}
\lim_{t\to\infty}e_{R_L}(t)=0,
\qquad
\lim_{t\to\infty}\hat R_L(t)=R_L.
\end{equation}
Moreover, since $\dot I_{L,\mathrm{ref}}=0$, the inductance adaptation law
satisfies $
\dot{\hat L}=0,$
so that $\hat L$ remains constant during operation at a constant current
reference.\footnote{Convergence of the inductance estimate to its true value
is not addressed in the present work. In principle, parameter convergence can
be obtained under additional excitation conditions, such as persistent
excitation. This would require sufficiently persistent excitation of $I_L$ and $\dot I_{L,\mathrm{ref}}$. The design of such excitation signals is outside
the scope of this work.}

\end{corollary}

\begin{proof}

For $\dot I_{L,\mathrm{ref}}(t)\equiv0$, the closed-loop error dynamics
\eqref{eq:error_dynamics} reduce to
\begin{align*}
\dot e_I
&=
-\frac{K_e}{L}e_I
+
\frac{I_L}{L}e_{R_L},\\
\dot e_{R_L}
&=
-\gamma_R I_L e_I,\\
\dot e_L
&=0.
\end{align*}
Since
\(
I_L=\bar I_{L,\mathrm{ref}}-e_I,
\)
the resulting error system is autonomous. Consider the positive definite and radially unbounded Lyapunov function
\[
V
=
\frac{1}{2}L e_I^2
+
\frac{1}{2\gamma_R}e_{R_L}^2
+
\frac{1}{2\gamma_L}e_L^2
\]
in $(e_I,e_{R_L},e_L)$. Differentiating $V$ along the reduced error
dynamics yields
\[
\dot V
=
-K_e e_I^2
\leq0.
\]
Hence, all trajectories are bounded and approach the largest invariant set
contained in
\[
\mathcal S
:=
\{(e_I,e_{R_L},e_L):\dot V=0\}
=
\{(e_I,e_{R_L},e_L):e_I=0\}.
\]
On $\mathcal S$, $e_I=0$, which implies
$I_L=\bar I_{L,\mathrm{ref}}$. Substituting this into the current-error
dynamics gives
\[
\dot e_I
=
\frac{\bar I_{L,\mathrm{ref}}}{L}\,e_{R_L}.
\]

For a trajectory to remain in $\mathcal S$, it must satisfy
$\dot e_I=0$. Since $\bar I_{L,\mathrm{ref}}\neq0$ and $L>0$,
this requires $e_{R_L}=0$. Hence, the largest invariant subset of
$\mathcal S$ is
\[
\mathcal M
=
\{(e_I,e_{R_L},e_L):e_I=0,\ e_{R_L}=0\},
\]
where $e_L$ may take any constant value.
By LaSalle's invariance principle \citep{khalil2002}, every trajectory
converges to $\mathcal M$. Consequently,
\[
\lim_{t\to\infty}e_I(t)=0,
\qquad
\lim_{t\to\infty}e_{R_L}(t)=0.
\]
Since $e_{R_L}=R_L-\hat R_L$, it follows that
\(
\lim_{t\to\infty}\hat R_L(t)=R_L.
\)
Finally, $\dot I_{L,\mathrm{ref}}=0$ implies from
\eqref{eq:L_adaptation} that
\(
\dot{\hat L}=0,
\)
and therefore $\hat L(t)\equiv\hat L(0)$.

\end{proof}

\section{Experimental Setup and Evaluation}
\label{sec:results}

This section describes the experimental platform and evaluates the proposed
adaptive current controller under plasma arc operation. The experimental
validation comprises two repeatable scenarios: contact-opening ignition and
current-reference tracking during established arc operation. In addition, two
qualitative case observations are presented: a short arc-voltage transient
during continued arc operation and one spontaneous arc extinction.

\subsection{Experimental Setup}

The experimental platform comprises the buck-converter power stage, current and voltage sensors, real-time control hardware, and an arc chamber with vertical electrode positioning. The chamber and positioning assembly are shown in Fig.~\ref{fig:experimental_setup_chamber}.

\paragraph{Power-stage implementation.}

The experimental setup implements the single-phase buck-converter topology
introduced in Subsection~\ref{subsec:power_stage} and shown in
Fig.~\ref{FIG:1}. The electrical ratings and component values of the
experimental converter are summarized in
Table~\ref{tab:experimental_parameters}. An additional snubber network is
connected across the switching stage to limit switching-induced voltage
transients.

\paragraph{Real-time control and data acquisition.}

The real-time control system consists of a Beckhoff industrial PC running
TwinCAT~3.1 and a TI C2000 F28P65 microcontroller connected as an EtherCAT
slave. The industrial PC coordinates the experiment, electrode motion,
reference generation, and data logging, while the microcontroller executes
the current controller with $T_s=125\,\mu\mathrm{s}$ and acquires
$I_L$, $U_{\mathrm{arc}}$, $U_S$, and $U_D$. The platform supports control
and switching frequencies up to approximately $200\,\mathrm{kHz}$, whereas
the employed IGBT limits the PWM frequency to $50\,\mathrm{kHz}$.\footnote{
Higher switching frequencies can be achieved using faster low-voltage
MOSFET or GaN transistors.}
The inductor current $I_L$ is measured using a fluxgate current sensor.
All measured electrical signals and the electrode gap $g$ are logged
synchronously on a common time base.

\begin{table}[t]
\centering
\caption{Converter, controller, and measurement parameters.}
\label{tab:experimental_parameters}
\small
\renewcommand{\arraystretch}{1.08}

\begin{tabular}{@{}p{0.61\columnwidth}p{0.30\columnwidth}@{}}
\toprule
Parameter & Value \\
\midrule

Nominal DC-link voltage $U_S$
& $40\,\mathrm{V}$ \\

Nominal inductance $L$
& $2.5\,\mathrm{mH}$ \\

Capacitance $C$
& $470\,\mathrm{nF}$ \\

Parallel resistance $R_C$
& $47\,\Omega$ \\

Power semiconductor
& Vishay IGBT \newline VS-GT55NA120  \\

PWM switching frequency $f_{\mathrm{PWM}}$
& $50\,\mathrm{kHz}$ \\

Controller sampling time $T_s$
& $125\,\mu\mathrm{s}$ \\

Controller hardware
& TI C2000 F28P65 \\

Supervisory control platform
& Beckhoff IPC, TwinCAT~3.1 \\

Current sensor
& LEM CKSR 50-NP fluxgate sensor \\

ADC resolution
& 16 bit \\

Feedback gain $K_e$
& 4 \\

Adaptation gains $\gamma_R$, $\gamma_L$
& 0.01 / 0.001 \\

Initial estimates $\hat R_L(0)$, $\hat L(0)$
& $500\,\mathrm{m}\Omega$ / $20\,\mathrm{m}\mathrm{H}$\\

\bottomrule
\end{tabular}

\end{table}

\paragraph{Plasma arc test chamber and electrode positioning.}

The aluminum test chamber is equipped with four optical side windows that
provide visual access to the discharge.\footnote{Two synchronized cameras
are arranged along crossed viewing directions, allowing reconstruction of
the spatial arc trajectory and estimation of the arc length
$\ell_{\mathrm{arc}}$. Under the investigated operating conditions, the
reconstructed arc length and the measured arc voltage $U_{\mathrm{arc}}$
exhibited closely correlated trends. The following evaluation therefore
focuses on the directly measured arc voltage; a separate optical analysis is
outside the scope of this paper.}
The lower contact plate is made of structural steel St52 and serves as the
grounded counter electrode. Argon is used as the process gas and can be
supplied through either the upper or lower part of the chamber.
\begin{figure*}[t]
\centering
\includegraphics[width=0.98\textwidth]{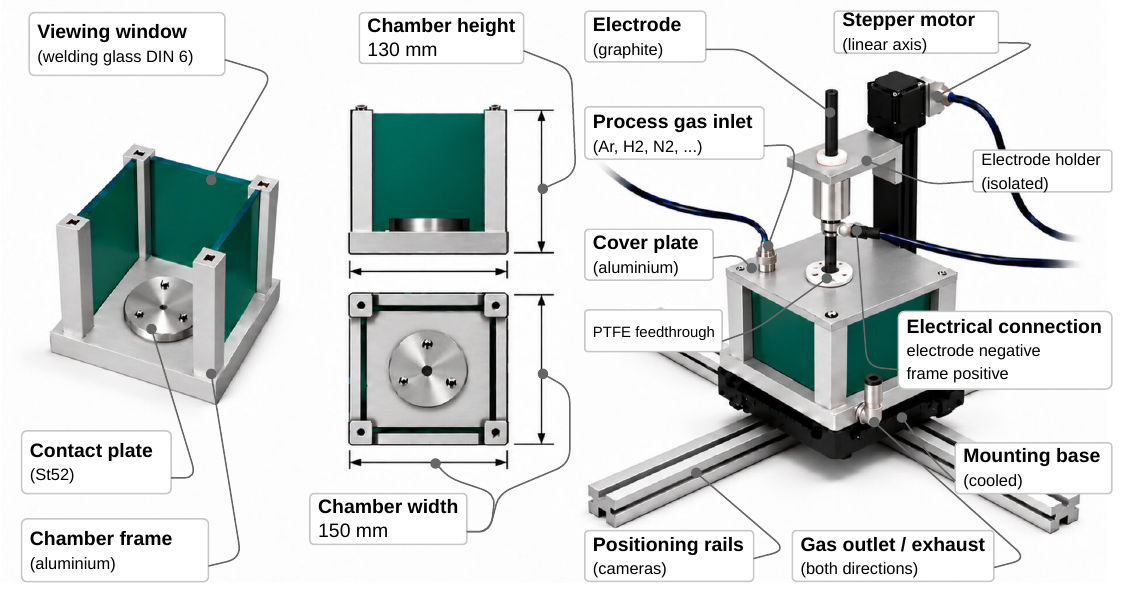}
\caption{Mechanical design of the experimental arc chamber. Left: chamber body with optical windows, contact plate, gas inlet, and principal dimensions. Right: assembled chamber with the vertical electrode-positioning axis, insulated electrode feedthrough, and gas connection.}
\label{fig:experimental_setup_chamber}
\end{figure*}
The graphite electrode is mounted on a vertical linear stage driven by a
stepper motor, whose commanded position defines the geometric electrode gap
$g$. A PTFE feedthrough electrically insulates the moving electrode holder
from the grounded chamber. The experiments use a solid graphite electrode
with a nominal diameter of $8\,\mathrm{mm}$. For contact-opening ignition, the electrode is first
brought into contact with the lower plate and subsequently retracted. A
spring beneath the holder limits the contact force and provides a moderate
contact resistance for local preheating before separation.

\subsection{Experimental Evaluation}

The following experiments evaluate the controller behavior during
contact-opening ignition and current-reference changes. Additional observations illustrate the response to a short arc-voltage
transient and to spontaneous arc extinction.

\subsubsection{Scenario A: Contact-Opening Ignition}

The first experiment evaluates the controller during contact-opening ignition.
This operating condition is particularly demanding because the electrical load
changes abruptly from a low-resistance contact to a plasma arc, accompanied by
a pronounced increase in $U_{\mathrm{arc}}$. The controller must accommodate
this load-voltage transient while maintaining the prescribed current.
Figure~\ref{fig:contact_ignition} shows the synchronized measurements during
the complete ignition sequence. The electrode is first brought into contact
with the lower plate, and a current of $15\,\mathrm{A}$ is established for
preheating. The electrode is then retracted until the conductive contact is
lost and a free-burning arc is formed.

\begin{figure*}[t]
    \centering
    \includegraphics[width=\linewidth]{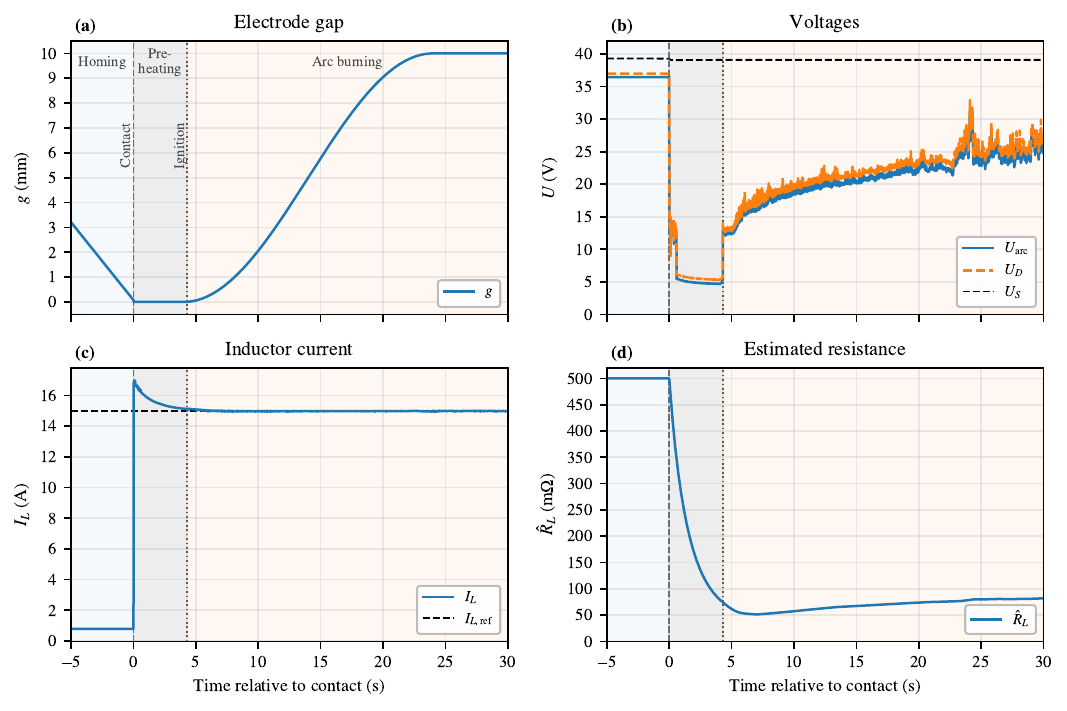}
    \caption{Measured contact-opening ignition sequence. (a) Electrode gap
    $g$ during homing, preheating, and arc operation. (b) Measured arc
    voltage $U_{\mathrm{arc}}$, converter output voltage $U_D$, and DC-link
    voltage $U_S$. (c) Inductor current $I_L$ and its reference
    $I_{L,\mathrm{ref}}=15\,\mathrm{A}$. (d) Online estimate of the
    effective series resistance $\hat R_L$. The dashed and dotted vertical
    lines indicate contact establishment and arc ignition, respectively.
    The shaded regions distinguish the homing, preheating, and arc-burning
    phases.}
    \label{fig:contact_ignition}
\end{figure*}

At the beginning of the preheating phase, $I_L$ exhibits a transient while
the resistance estimate $\hat R_L$ adapts from its initial value. As
$\hat R_L$ approaches its steady value, the current converges to its
reference. During electrode retraction and arc formation,
$U_{\mathrm{arc}}$ increases substantially, whereas $I_L$ remains close to
$15\,\mathrm{A}$. The measured response therefore shows that the arc-voltage
feedforward together with the adaptive current controller maintains current
regulation during the transition from contact conduction to free-burning arc
operation.
Figure~\ref{fig:contact_ignition}(d) shows the corresponding online estimate
$\hat R_L$ of the effective series resistance. The estimate adapts rapidly
during the initial current buildup and subsequently remains close to its
settled value throughout the contact-opening sequence. The inductance estimate
$\hat L$ is not shown, since its adaptation law \eqref{eq:L_adaptation} is
driven by $\dot I_{L,\mathrm{ref}}$ and therefore becomes inactive once the
current reference is constant.
Short fluctuations of $U_{\mathrm{arc}}$ occur during the initial contact
phase and are attributed to intermittent electrical contact associated with
the electrode surfaces and the spring-loaded holder. A further voltage
disturbance is observed at approximately $t=24\,\mathrm{s}$ after the
electrode-positioning motor is switched off. 
The temporary perturbation of the plasma arc is likely caused by a transient change in the external magnetic field associated with the decaying motor current. Despite this effect, \(I_L\) remains close to its reference. Similar disturbances
with varying magnitude were observed in additional experiments.

Overall, the experiment demonstrates successful contact-opening ignition and
current regulation despite the abrupt transition from contact conduction to a
free-burning arc and the associated arc-voltage variations.

\subsubsection{Scenario B: Current-Reference Tracking}

The second experiment evaluates current-reference tracking during established
arc operation. In contrast to the ignition experiment, the current reference
is varied while a stable plasma arc is already present. The reference changes
also activate the inductance adaptation through
$\dot I_{L,\mathrm{ref}}$ in \eqref{eq:L_adaptation}.

Figure~\ref{fig:current_tracking} shows the measured response to a sequence of
current-reference changes. The inductor current $I_L$ follows the prescribed
reference over the investigated operating range. The corresponding variations
in $U_{\mathrm{arc}}$ reflect the nonlinear voltage-current characteristic and external disturbances
of
the plasma arc, while the converter voltage $U_D$ adjusts accordingly.

After the current-reference sequence, the electrode gap is increased from
approximately $10\,\mathrm{mm}$ to $20\,\mathrm{mm}$. This results in an increase in $U_{\mathrm{arc}}$ and consequently in the
required converter voltage, while the current remains regulated. The experiment thus
combines current-reference variations with an additional change in the plasma
load.

The corresponding inductance estimate $\hat L$ is shown in
Fig.~\ref{fig:current_tracking}(d). According to
\eqref{eq:L_adaptation}, the estimate changes during current-reference
transitions and remains constant when $\dot I_{L,\mathrm{ref}}=0$.
Figure~\ref{fig:current_tracking_detailed} provides a detailed view of one
such transition and illustrates the simultaneous current response and update
of $\hat L$. The estimated value should be interpreted as the inductance
parameter of the averaged converter model under the available excitation;
convergence to the physical inductance is not inferred from this experiment.

\begin{figure*}[t]
    \centering

    \includegraphics[width=\linewidth]
    {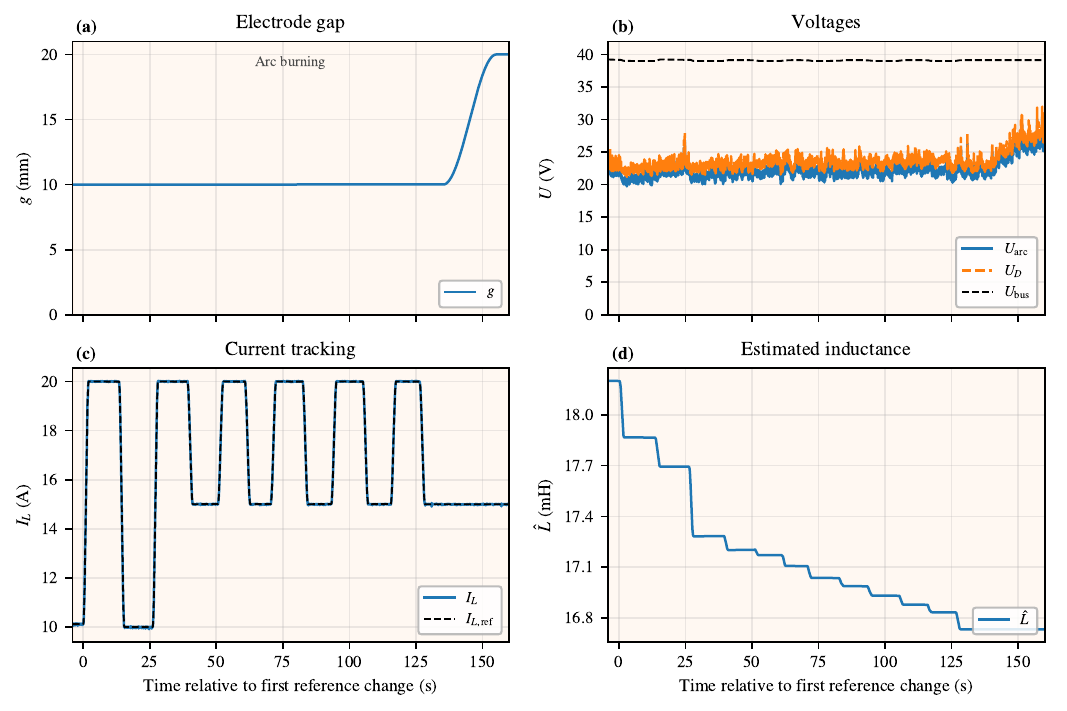}

    \caption{Experimental current-reference tracking during established arc
    operation. (a) Electrode gap $g$. (b) Measured arc voltage
    $U_{\mathrm{arc}}$, converter output voltage $U_D$, and DC-link voltage
    $U_S$. (c) Measured inductor current $I_L$ and its reference
    $I_{L,\mathrm{ref}}$. (d) Online adaptation of the inductance $\hat L$.
    Time is referenced to the beginning of the first current-reference
    transition.}

    \label{fig:current_tracking}

    \vspace{0.8em}

    \includegraphics[width=\linewidth]
    {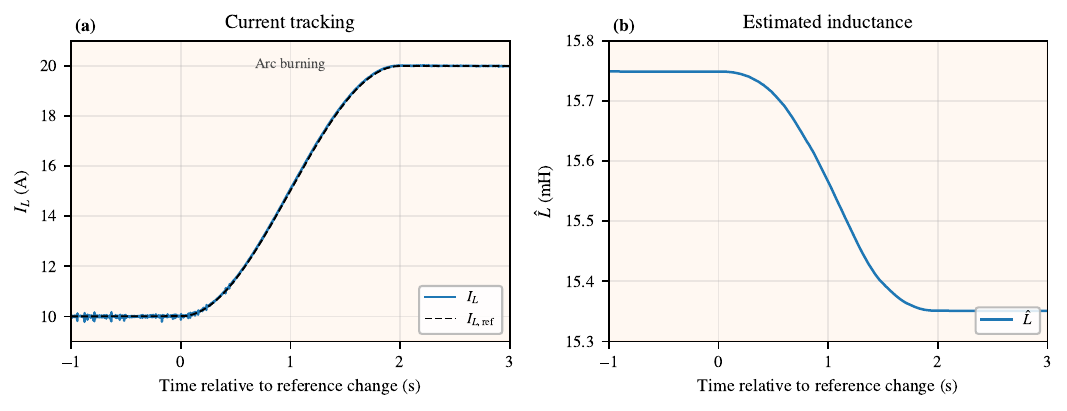}

    \caption{Detailed response to a single current-reference transition during
    established arc operation. (a) Measured inductor current $I_L$ and its
    reference $I_{L,\mathrm{ref}}$. (b) Corresponding online adaptation of the
    inductance $\hat L$. Time is referenced to the beginning of the
    current-reference transition.}

    \label{fig:current_tracking_detailed}
\end{figure*}
The detailed measurements further indicate that the arc exhibits noticeably
larger fluctuations at the lower current level of $10\,\mathrm{A}$ than at
$20\,\mathrm{A}$. The same tendency is observed in the measured arc voltage,
although the corresponding voltage signal is not shown in
Fig.~\ref{fig:current_tracking_detailed}. A possible explanation is the lower
power input at reduced current, resulting in weaker plasma ionization and a
greater sensitivity of the arc to variations in electrode gap and external
disturbances. The observation suggests that the current reference should not
only be selected according to the desired process conditions, but also with
respect to arc stability. This motivates a higher-level operating-point
selection that avoids operating conditions with increased fluctuations and,
potentially, an increased risk of arc extinction.

\subsubsection{Additional Experimental Observations: Arc-Voltage Transient and Arc Extinction}

Figure~\ref{fig:arc_voltage_transient} shows a short arc-voltage transient
recorded during established arc operation. Such transients occur when the
plasma arc relocates and reattaches at a different position on the electrode
surface, causing a rapid change in the effective arc length
$\ell_{\mathrm{arc}}$. The resulting abrupt variation of
$U_{\mathrm{arc}}$ is accompanied by a corresponding change in the required
converter voltage $U_D$, while the inductor current remains close to its
reference of $15\,\mathrm{A}$. The event therefore illustrates that current
regulation is maintained despite a rapid change in arc geometry and arc
voltage.

In a separate experiment, a spontaneous complete arc extinction was observed.
The corresponding measurements are shown in
Fig.~\ref{fig:arc_extinction}. Starting from
$I_{L,\mathrm{ref}}=20\,\mathrm{A}$, the required converter voltage increases
until the power stage reaches its voltage limit of $37\,\mathrm{V}$ at
approximately $t=-10\,\mathrm{ms}$. From this point onward, the converter can
no longer provide the voltage required to maintain the prescribed current and
$I_L$ begins to decrease.

During the subsequent current decay, $U_{\mathrm{arc}}$ continues to increase
and temporarily exceeds the converter output voltage. The magnetic energy
stored in the inductor continues to drive current through the arc even though
the converter is already operating at its voltage limit. As $I_L$ decreases,
the stored magnetic energy is progressively released. At $t=0$, the remaining
energy is no longer sufficient to sustain arc conduction and the arc
extinguishes.

The observation also highlights the influence of the operating current on the
ability to maintain arc conduction during short high-voltage transients.
Since the magnetic energy stored in the inductor is proportional to
$L I_L^2$, a higher current operating point provides a larger energetic
reserve. If the converter reaches its voltage limit, this additional stored
energy can sustain the arc current for a longer period while $I_L$ decays.
Higher-current operating points can therefore improve the capability to bridge
short transient increases in arc voltage. A larger inductance or a higher
available converter voltage can further extend the duration for which such
conditions can be tolerated.

\begin{figure*}[t]
    \centering

    \includegraphics[width=\linewidth]
    {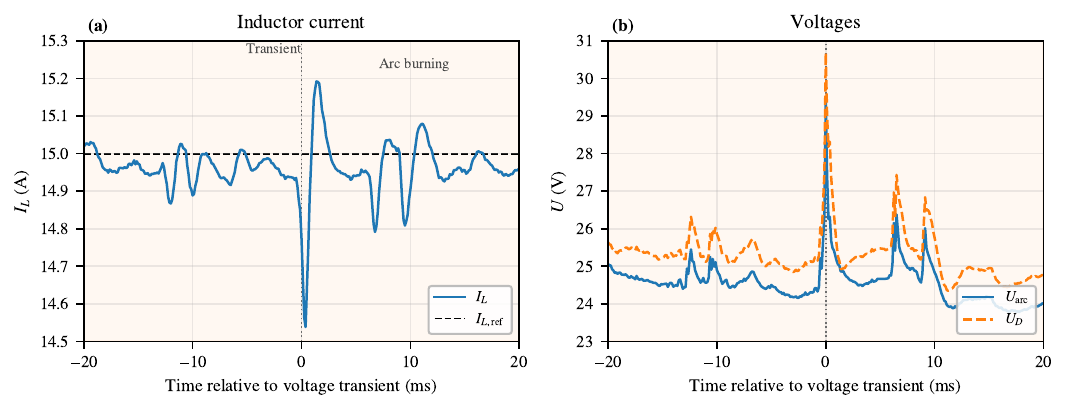}

    \caption{Measured arc-voltage transient during established arc operation.
    (a) Inductor current $I_L$ and its reference
    $I_{L,\mathrm{ref}}=15\,\mathrm{A}$. (b) Measured arc voltage
    $U_{\mathrm{arc}}$ and converter output voltage $U_D$. The transient is
    associated with a rapid relocation of the arc attachment point and the
    resulting change in effective arc length.}

    \label{fig:arc_voltage_transient}

    \vspace{0.8em}

    \includegraphics[width=\linewidth]
    {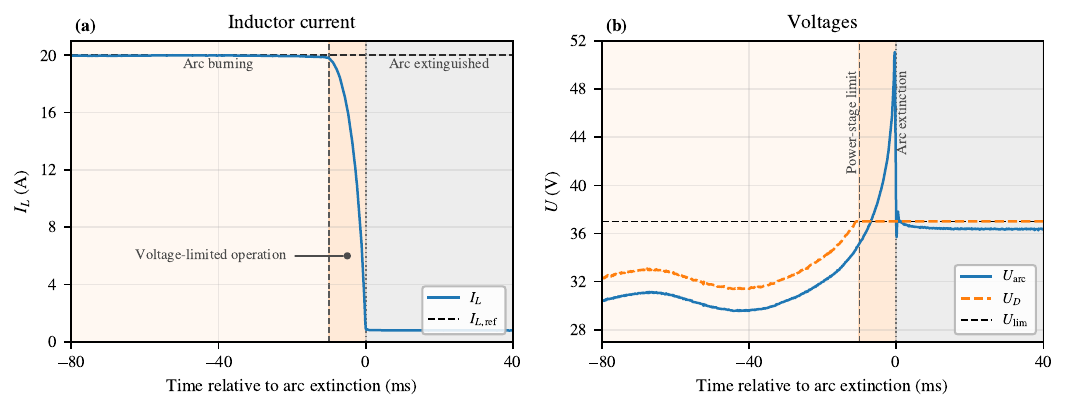}

    \caption{Measured spontaneous arc extinction. (a) Inductor current $I_L$
    and its reference $I_{L,\mathrm{ref}}=20\,\mathrm{A}$. (b) Measured arc
    voltage $U_{\mathrm{arc}}$, converter output voltage $U_D$, and
    power-stage voltage limit $U_{\mathrm{lim}}=37\,\mathrm{V}$. The dashed
    vertical line indicates the onset of voltage limitation at
    $t=-10\,\mathrm{ms}$, and the dotted line at $t=0$ marks arc extinction.}

    \label{fig:arc_extinction}
\end{figure*}

These observations motivate a supervisory operating-point selection that
considers arc stability in addition to the desired process current and the
available converter voltage. The existence of operating-point-dependent
stability fields for HPSR plasma arcs has also been reported by
Zarl et al.~\cite{Zarl2020}. A systematic characterization of such stability
limits for the present setup is left for future work.

\section{Conclusions and Outlook}
\label{sec:conclusion}

An adaptive average current-mode controller has been developed for a buck
converter supplying a nonlinear plasma arc load. The control law combines
current-error feedback with feedforward compensation of the measured arc
voltage and online adaptation of the inductance and its effective series
resistance.

For the ideal unsaturated averaged closed-loop system, a Lyapunov analysis
establishes boundedness of the tracking and parameter-error signals and
asymptotic convergence of the current tracking error. The experimental results
demonstrate successful contact-opening ignition and current-reference tracking
during established arc operation. Current regulation is also maintained during
rapid arc-voltage variations caused by changes in the arc geometry. The
spontaneous arc-extinction experiment further illustrates the practical
limitation imposed by the available converter voltage.

Two extensions of the proposed control architecture are of particular
interest.

\subsection{Multiphase Buck-Converter Extension}

Higher arc currents can be supplied by operating several interleaved
buck-converter phases in parallel. The proposed adaptive controller can in principle be
applied to each phase, while a supervisory layer distributes the total current
reference among the active phases. Future work will address current sharing,
synchronization, phase-dependent parameter adaptation, and operation following
the activation or loss of individual phases.

\subsection{Supervisory Arc-Stability Management}

The experiments indicate that arc stability depends on the selected current
operating point and the available voltage margin. Increased fluctuations were
observed at lower current levels, while arc extinction occurred when the
required converter voltage exceeded the available range. This motivates a
supervisory layer that considers both process requirements and stable arc
operation.

Short current pulses may provide an additional means of stabilizing the
discharge during critical operating conditions. The proposed controller is
well suited for this extension because such pulses can be introduced directly
through the current reference $I_{L,\mathrm{ref}}$. Future work will focus on
the identification of arc-stability boundaries and the experimental evaluation
of suitable stabilization strategies.
\printcredits

\section*{Declaration of generative AI and AI-assisted technologies in the writing process}

During the preparation of this work, the authors used DeepL for translation and linguistic refinement and generative AI tools to improve the readability and language of the manuscript and to refine selected graphical elements of the experimental setup figures. After using these tools, the authors reviewed and edited the content as needed and take full responsibility for the content of the publication.
\bibliographystyle{cas-model2-names}

\bibliography{cas-refs}

@book{khalil2002,
  author    = {Khalil, Hassan K.},
  title     = {Nonlinear Systems},
  edition   = {3},
  publisher = {Prentice Hall},
  year      = {2002}
}

@article{sanders1991,
  author  = {Sanders, Seth R. and Verghese, George C.},
  title   = {Synthesis of Averaged Circuit Models for Switched Power Converters},
  journal = {IEEE Transactions on Circuits and Systems},
  volume  = {38},
  number  = {8},
  pages   = {905--915},
  year    = {1991},
  doi     = {10.1109/31.85632}
}

@book{Gerdeman1972,
  author    = {Gerdeman, D. A. and Hecht, N. L.},
  title     = {Arc Plasma Technology in Materials Science},
  publisher = {Springer},
  address   = {Vienna},
  year      = {1972},
  doi       = {10.1007/978-3-7091-8293-2}
}

@article{Hassan2025,
  author  = {Hassan, Islam and Sayed, Khairy and Alanazi, Meshari D.
             and Ziedan, Hamdy A. and Abdelsattar, Montaser},
  title   = {A Comprehensive Review of Power Converter Topologies
             and Control Techniques for Welding Power Supply Applications},
  journal = {Discover Applied Sciences},
  volume  = {7},
  pages   = {1036},
  year    = {2025},
  doi     = {10.1007/s42452-025-07180-1}
}

@article{Ernst2023,
  author  = {Ernst, Daniel and Zarl, Michael Andreas and
             Farkas, Manuel Andreas and Schenk, Johannes},
  title   = {Effects of the Electrodes' Shape and Graphite Quality
             on the Arc Stability During Hydrogen Plasma Smelting
             Reduction of Iron Ores},
  journal = {Steel Research International},
  volume  = {94},
  pages   = {2200818},
  year    = {2023},
  doi     = {10.1002/srin.202200818}
}

@article{Solanki2015,
  author  = {Solanki, Jayant and Fr{\"o}hleke, Norbert and B{\"o}cker, Joachim},
  title   = {High-Current Variable-Voltage Rectifiers:
             State of the Art Topologies},
  journal = {IET Power Electronics},
  volume  = {8},
  number  = {6},
  pages   = {1068--1080},
  year    = {2015},
  doi     = {10.1049/iet-pel.2014.0533}
}

@article{AlNasser2026,
  author  = {Al-Nasser, Mohamad and Karimi-Sibaki, Ebrahim
             and Alloush, Mhamad Mahdi and Vakhrushev, Alexander
             and Wu, Menghuai and Hackl, Gernot and Ishmurzin, Anton
             and Redl, Christian and Kharicha, Abdellah},
  title   = {Quantitative Impact of Hydrogen on Voltage Characteristics
             and Arc Stability in Industrial {EAFs} and {HPSR}},
  journal = {Results in Engineering},
  volume  = {31},
  pages   = {111480},
  year    = {2026},
  doi     = {10.1016/j.rineng.2026.111480}
}

@article{Nizami2021,
  author  = {Nizami, Tousif Khan and Chakravarty, Arghya and Mahanta, Chitralekha},
  title   = {Time Bound Online Uncertainty Estimation Based Adaptive
             Control Design for {DC--DC} Buck Converters with
             Experimental Validation},
  journal = {IFAC Journal of Systems and Control},
  volume  = {15},
  pages   = {100127},
  year    = {2021},
  doi     = {10.1016/j.ifacsc.2020.100127}
}

@article{Saadat2022,
  author  = {Saadat, Seyyed Amirhossein and
             Ghamari, Seyyed Morteza and
             Mollaee, Hasan},
  title   = {Adaptive Backstepping Controller Design on Buck Converter
             with a Novel Improved Identification Method},
  journal = {IET Control Theory \& Applications},
  volume  = {16},
  pages   = {485--495},
  year    = {2022},
  doi     = {10.1049/cth2.12241}
}

@article{ElFadil2009,
  author  = {El Fadil, Hassan and Giri, Fouad},
  title   = {Robust Nonlinear Adaptive Control of Multiphase
             Synchronous Buck Power Converters},
  journal = {Control Engineering Practice},
  volume  = {17},
  number  = {11},
  pages   = {1245--1254},
  year    = {2009},
  doi     = {10.1016/j.conengprac.2009.04.009}
}

@book{Beilis2020,
  author    = {Beilis, Isak I.},
  title     = {Plasma and Spot Phenomena in Electrical Arcs},
  series    = {Springer Series on Atomic, Optical, and Plasma Physics},
  volume    = {113},
  publisher = {Springer},
  year      = {2020},
  doi       = {10.1007/978-3-030-44747-2}
}

@article{Haraldsson2025,
  author  = {Haraldsson, Hakon Valur and
             Reynolds, Quinn and
             Tesfahunegn, Yonatan A. and
             Saevarsdottir, Gudrun},
  title   = {Modeling of Industrial Electric Arcs Using Different
             Plasma Gas Compositions},
  journal = {Metallurgical and Materials Transactions B},
  volume  = {56},
  pages   = {1208--1217},
  year    = {2025},
  doi     = {10.1007/s11663-024-03397-4}
}

@inproceedings{VoglhuberBrunnmaier2026,
  author    = {Voglhuber-Brunnmaier, Thomas and Lehner, Bernhard and Stanek, Jan and Ecker, Lukas and Reichel, Erich K.},
  title     = {Instrumentierung und Stabilisierung von Lichtbögen zur Plasma-Schmelzreduktion},
  booktitle = {Vorträge},
  year      = {2026},
  pages     = {343--350}
}

@article{Zarl2020,
  author  = {Zarl, Michael Andreas and
             Farkas, Manuel Andreas and
             Schenk, Johannes},
  title   = {A Study on the Stability Fields of Arc Plasma
             in the {HPSR} Process},
  journal = {Metals},
  volume  = {10},
  number  = {10},
  pages   = {1394},
  year    = {2020},
  doi     = {10.3390/met10101394}
}

@article{Zhang2025Interleaved,
  author  = {Zhang, Xiuyun and Zhang, Guidong and Yu, Samson S.},
  title   = {Review of Control Techniques for Interleaved Buck Converters:
             Control Strategies, Efficiency Optimization and Phase Shedding},
  journal = {Chinese Journal of Electrical Engineering},
  volume  = {11},
  number  = {1},
  pages   = {40--58},
  year    = {2025},
  doi     = {10.23919/CJEE.2025.000108}
}

@article{Li2011CurrentSharing,
  author  = {Li, Ming and Tse, Chi K. and Ma, Xikui},
  title   = {Calculation of Steady-State Solution of Parallel-Connected
             Buck Converters with Active Current Sharing and Its
             Parameter Sensitivity},
  journal = {International Journal of Circuit Theory and Applications},
  volume  = {39},
  number  = {3},
  pages   = {275--297},
  year    = {2011},
  doi     = {10.1002/cta.633}
}

@article{Sinafar2023,
  author  = {Sinafar, Behzad and Badamchizadeh, Mohammad Ali and
             Kharrati, Hamed and Baradarannia, Mahdi},
  title   = {Current Sharing and Voltage Regulation of Parallel
             {DC--DC} Buck Converters: Switching Control Approach},
  journal = {ISA Transactions},
  volume  = {140},
  pages   = {490--502},
  year    = {2023},
  doi     = {10.1016/j.isatra.2023.05.022}
}

@inproceedings{andrea2015arc,
  author    = {Andrea, Jonathan and Besdel, Patrick and
               Zirn, Olivier and Bournat, Marc},
  title     = {The Electric Arc as a Circuit Component},
  booktitle = {IECON 2015 -- 41st Annual Conference of the
               IEEE Industrial Electronics Society},
  address   = {Yokohama, Japan},
  pages     = {3027--3034},
  year      = {2015},
  publisher = {IEEE},
  doi       = {10.1109/IECON.2015.7392564}
}

@article{warrington1931reactance,
  author  = {Warrington, A. R. van C.},
  title   = {Reactance Relays Negligibly Affected by Arc Impedance},
  journal = {Electrical World},
  volume  = {98},
  number  = {12},
  pages   = {502--505},
  year    = {1931}
}

@book{warrington1968protective,
  author    = {Warrington, A. R. van C.},
  title     = {Protective Relays: Their Theory and Practice},
  edition   = {2},
  volume    = {1},
  publisher = {Chapman and Hall},
  address   = {London},
  year      = {1968}
}





\end{document}